\documentclass[11pt]{amsart}
\usepackage{graphicx}
\usepackage[active]{srcltx}
 \makeatletter
\renewcommand*\subjclass[2][2000]{%
  \def\@subjclass{#2}%
  \@ifundefined{subjclassname@#1}{%
    \ClassWarning{\@classname}{Unknown edition (#1) of Mathematics
      Subject Classification; using '1991'.}%
  }{%
    \@xp\let\@xp\subjclassname\csname subjclassname@#1\endcsname
  }%
}
 \makeatother
\usepackage{enumerate,url,amssymb,  mathrsfs,yhmath}%, pdfsync}% ,showkeys, pdfsync}

\newtheorem{theorem}{Theorem}[section]

\newtheorem*{lemma*}{Lemma}
\newtheorem{proposition}[theorem]{Proposition}
\newtheorem{corollary}[theorem]{Corollary}
\newcommand{\onto}{\xrightarrow[]{{}_{\!\!\textnormal{onto}\!\!}}}

\theoremstyle{definition}

\newtheorem{example}[theorem]{Example}

\theoremstyle{remark}
\newtheorem{remark}[theorem]{Remark}

\numberwithin{equation}{section}

\def\XXint#1#2#3{{\setbox0=\hbox{$#1{#2#3}{\int}$}
\vcenter{\hbox{$#2#3$}}\kern-.5\wd0}}

\def\le{\leqslant}
\def\ge{\geqslant}
\begin{document}

\title{A sharp inequality for harmonic diffeomorphisms of the unit disk} \subjclass{Primary 31R05;
Secondary 42B30 }

%\date{11 October, 2005}

\author{David Kalaj}
\address{University of Montenegro, Faculty of Natural Sciences and
Mathematics, Cetinjski put b.b. 81000 Podgorica, Montenegro}
\email{davidk@ac.me}

\begin{abstract}
We extend the classical Schwarz-Pick inequality to the class of harmonic mappings between the unit disk and a Jordan domain with given perimeter. It is intriguing that the extremals in this case are certain harmonic diffeomorphisms between the unit disk and a convex domain that solve the Beltrami equation of second order.
\end{abstract}  \maketitle

\section{Introduction}

Let $\mathbf{U}$ be the unit disk in the complex plane $\mathbf{C}$ and denote by $\mathbf{T}$ its boundary. A harmonic mapping $f$ of the unit disk into the complex plane can be written by $f(z)=g(z)+\overline{h(z)}$ where $g$ and $h$ are holomorphic functions defined on the unit disk.  Two of essential properties of harmonic mappings are given by Lewy theorem, and Rado-Kneser-Choquet theorem. Lewy theorem states that a injective harmonic mapping is indeed a diffeomorphism. Rado-Kneser-Choquet theorem states that a Poisson extension of a homeomorphism of the unit circle $\mathbf{T}$ onto a convex Jordan curve $\gamma$ is a diffeomorphism on the unit disk onto the inner part of $\gamma$. For those and many more important properties of harmonic mappings we refer to the book of Duren \cite{duren}.

The standard Schwarz-Pick lemma for holomorphic mappings states that every holomorphic mapping $f$ of the unit disk onto itself satisfies the inequality \begin{equation}\label{schar}
|f'(z)|\le \frac{1-|f(z)|^2}{1-|z|^2}.
\end{equation}
If the equality is attained in \eqref{schar} for a fixed $z=a\in\mathbf{U}$, then $f$ is a M\"obius transformation of the unit disk.

It follows from \eqref{schar} the weaker inequality  \begin{equation}\label{schar1}
|f'(z)|\le \frac{1}{1-|z|^2}
\end{equation}
with the equality in \eqref{schar1} for some fixed $z=a$ if and only if $f(z)=e^{it}\frac{z-a}{1-z\bar a}$.
We will extend this result to harmonic mappings.

\section{Main result}

\begin{theorem}\label{prev}
If $f$ is a harmonic orientation preserving diffeomorphism of the unit disk $\mathbf{U}$ onto a Jordan domain $\Omega$ with rectifiable boundary of length $2\pi R$, then the sharp inequality
\begin{equation}\label{partibar}|\partial f(z)|\le \frac{R}{1-|z|^2},\ \ \ z\in\mathbf{U}\end{equation} holds. If the equality in \eqref{partibar} is attained for some $a$, then $\Omega$ is convex and there is a holomorphic function $\mu:\mathbf{U}\to\mathbf{U}$ and a constant $\theta\in[0,2\pi]$ such that

\begin{equation}\label{convex}F(z):=e^{-i\theta}f\left(\frac{z+a}{1+z\bar a}\right)=R\left(\int_0^z\frac{ dt}{1+t^2\mu(t)}+\overline{\int_0^z\frac{\mu(t)dt}{1+t^2\mu(t)}}\right).\end{equation}

Moreover every function $f$ defined by  \eqref{convex}, is a harmonic diffeomorphism and maps the unit disk to a  Jordan domain bounded by a convex curve of length $2\pi R$ and the inequality \eqref{partibar} is attained for $z=a$.

\end{theorem}
\begin{corollary}
Under the conditions of Theorem~\ref{prev}, if $|\mu |_{\infty}=k<1$, then the mapping $F$ is $K=\frac{1+k}{1-k}$ bi-Lipschitz, and $K-$quasiconformal.
\end{corollary}
\begin{proof}
We have that $$f_z(z)=\frac{1}{1+z^2\mu(z)}$$ and $$\overline{f_{\bar z}(z)}= \frac{\mu(z)}{1+z^2\mu(z)}.$$ Thus
$$\frac{1-k}{1+k}\le |F_z|-|F_{\bar z}|=|lF|\le |dF|=|F_z|+|F_{\bar z}|\le \frac{1+k}{1-k}.$$
\end{proof}

\begin{corollary}\label{marte}  If $\Omega=\mathbf{U}$, then the equality is attained in \eqref{partibar} for some $a$ if only if $f$ is a M\"obius transformation of the unit disk onto itself.
\end{corollary}
\begin{proof}[Proof of Corollary~\ref{marte}]
Under conditions of Theorem~\ref{prev} the function \eqref{convex} can be written as \begin{equation}\label{convex1}F(z):=e^{-i\theta}f\left(\frac{z+a}{1+z\bar a}\right)=R\left(\int_0^z\left(1-t^2 h'(t)\right) dt+\overline{h(z)}\right)\end{equation} where
$h(z)=\sum_{k=0}^\infty a_k z^k$ is defined on the unit disk and satisfies the condition
\begin{equation}\label{eqp} \frac{|h'(z)|}{|1-z^2 h'(z)|}< 1, \ \ z\in\mathbf{U}.\end{equation}
 Moreover $$|f(\mathbf{U})|=R^2\pi\left(1- \sum_{k=0}^\infty \frac{|a_k|^2}{(n+1)(n+2)}\right). $$
If $R=1$,  this implies that $\Omega=\mathbf{U}$ if and only if $h\equiv 0$.
 \end{proof}
 By using the corresponding result in \cite{cs} and Theorem~\ref{prev} we have
\begin{corollary}
If as in \eqref{convex1}, $F(z)=g+\overline{h}$, then $F(z)=g(z)-h(z)$ is univalent and convex in direction of real axis.
\end{corollary}
By using Theorem~\ref{prev} we obtain
\begin{corollary}
For every positive constant $R$  and every holomorphic function $\mu$ of the unit disk into itself, there is a unique convex Jordan domain $\Omega=\Omega_{\mu,R}$, with the perimeter $2\pi R$ , such that the initial boundary problem
\begin{equation}\label{boundary}\left\{
  \begin{array}{ll}
    \overline{f_{\bar z}(z)}=\mu(z) f_z(z), & \hbox{} \\
    f_z(0)=R, & \hbox{}
 \\
    f(0)=0. & \hbox{}
  \end{array}
\right.
\end{equation}
admits a unique univalent harmonic solution $f=f_{\mu,R}:\mathbf{U}\onto \Omega$.
\end{corollary}
\begin{remark}
If instead of boundary problem \eqref{boundary} we observe  \begin{equation}\label{boundary1}\left\{
  \begin{array}{ll}
    \overline{g_{\bar z}(z)}=\mu(z) g_z(z), & \hbox{} \\
    g_z(a)=\frac{R}{1-|a|^2}, & \hbox{}
 \\
    g(a)=0, & \hbox{}
  \end{array}
\right.
\end{equation}
then the solution $g$ is given by $$g(z)=e^{i\theta}f(\frac{z-a}{1-z\bar a})$$ and thus $g(\mathbf{U})=e^{i\theta} \cdot \Omega_{\mu, R}$. Here $f$ is a solution of \eqref{boundary}.
\end{remark}

\section{Proof of the main result}
\begin{proof}[Proof of Theorem~\ref{prev}]
Assume first that $f(z)=g(z)+\overline{h(z)}$ has $C^1$ extension to the boundary and assume without loos of generality that $R=1$. Then we have $$\partial_t \left({g(z)+\overline{h(z)}}\right)=ig'(z) z+\overline{i h'(z) z}$$
So for $z=e^{it}$, $$|ig'(z) z+\overline{i h'(z) z}|=|g'(z)-\overline{h'(z) z^2}|.$$
Thus
$$2\pi=\int_{\mathbf{T}} \left|\partial_t \left({g(z)+\overline{h(z)}}\right)\right| |dz|=\int_{\mathbf T}|g'(z)-\overline{h'(z) z^2}||dz|.$$

As $ |g'(z)-\overline{h'(z) z^2}|$ is subharmonic, it follows that $$|g'(0)|\le \frac{1}{2\pi}\int_{\mathbf T}|g'(z)-\overline{h'(z) z^2}||dz|.$$

Thus we have that $|g'(0)|\le 1$.  Now if $m(z)=\frac{z+a}{1+z a}$, then $m(0)=a$, and thus $F(z)=f(m(z))$ is a harmonic diffeomorphism of the unit disk onto itself. Further, $$\partial F(0)=f'(a)m'(0)=\partial f(a)(1-|a|^2).$$ Therefore by applying the previous case to $F$ we obtain $$|\partial f(a)|\le \frac{1}{1-|a|^2}.$$

Assume now that the equality is attained for $z=0$. Then
$$|g'(0)|= \frac{1}{2\pi r}\int_{\mathbf {rT}}|g'(z)-\overline{h'(z) z^2}||dz|,$$ or what is the same

$$|g'(0)|= \frac{1}{2\pi }\int_{\mathbf {T}}|g'(zr)-\overline{h'(zr) r^2z^2}||dz|.$$

Thus for $0\le r\le 1$ we have \begin{equation}\label{equ}\frac{1}{2\pi}\int_{\mathbf {T}}|g'(zr)-\overline{h'(rz) r^2z^2}||dz|-|g'(0)|\equiv 0.\end{equation}

In order to continue recall the definition of the Riesz measure $\mu$ of a subharmonic function $u$. Namely there exists a unique positive Borel measure $\mu$ so that $$\int_{\mathbf U} \varphi (z) d\mu(z)=\int_{\mathbf{U}} u \Delta \varphi(z) dm(z),\ \ \ \varphi\in C_0^2(\mathbf{U}).$$

Here $dm$ is the Lebesgue measure defined on the complex plane $\mathbf{C}$.
If $u\in C^2$, then $$d\mu=\Delta u dm.$$

\begin{proposition}\cite[Theorem~4.5.1]{pavlo}\label{pavlo}
If $u$ is a subharmonic function defined on the unit disk then for $r<1$ we have

$$\frac{1}{2\pi}\int_{\mathbf T} u(r z)|dz|-u(0)=\frac{1}{2\pi}\int_{|z|<r} \log\frac{r}{|z|} d\mu(z)$$ where $\mu$ is the Riesz measure of $u$.
\end{proposition}
By applying Proposition~\ref{pavlo} to the subharmonic function $$u(z)=|g'(z)-\overline{h'(z) z^2}|$$ in view of \eqref{equ} we obtain that $$\frac{1}{2\pi}\int_{|z|<r} \log\frac{r}{|z|} d\mu(z)\equiv 0.$$ Thus in particular we infer that $\mu=0$, or what is the same $\Delta u=0$. As $u=|w|$ where $w=|u|e^{i\theta}$ is harmonic, it follows that
$$\Delta u = u|\nabla \theta|^2=0.$$ Therefore $\nabla \theta\equiv 0$, and therefore $\theta=\mathrm{const}$.

So $$e^{-i\theta}(g'(z)-\overline{h'(z) z^2})=G(z)+\overline{H(z)},$$ is a real harmonic function. Here  $$G(z)=e^{-i\theta}g'(z)$$ and $$H(z)=-e^{i\theta}{h'(z) z^2}$$ are analytic functions satisfying the condition $|H(z)|<|G(z)|$ in view of Lewy theorem. Thus $$G(z)+\overline{H(z)}=\overline{G(z)}+H(z)$$ or what is the same
  $$G(z)-H(z)=\overline{G(z)-H(z)}.$$ Thus $G(z)-H(z)$ is a real holomorphic function and therefore it is a constant function. Further $$e^{-i\theta}g'(z)+e^{i\theta}{h'(z) z^2}=G(z)-H(z)=G(0)-H(0)=e^{-i\theta}g'(0).$$

  Hence $$G(z)+\overline{H(z)}=G(z)+\overline{G(z)}-e^{-i\theta}g'(0)=2\Re \left[e^{-i\theta}g'(z)\right]-e^{-i\theta}g'(0).$$

Assume  without losing the  generality  that $\theta=0$ and $g'(0)=1$. Then \begin{equation}\label{after}g'(z)=1-{h'(z) z^2}.\end{equation}

Further for $z=e^{it}$,  $$\partial_t F(z)=i z(1-2\Re(h'(z) z^2))$$ and  $$|\partial_t f(z)|=|g'(z)-\overline{h'(z) z^2}|=|1-2\Re(h'(z) z^2)|=1-2\Re(h'(z) z^2).$$

 From \eqref{eqp}, we infer that $$(1-2\Re(h'(z) z^2))>|h'(z)|^2(1-|z|^2).$$
In order to get the representation \eqref{convex}, by Lewy theorem, we have that the holomorphic mapping $\mu(z)=\frac{h'(z)}{g'(z)}$ maps the unit disk into itself. By \eqref{after} we deduce that $$g(z)=R\int_0^z\frac{ dt}{1+t^2\mu(t)}$$ and $$h(z)=R{\int_0^z\frac{\mu(t)dt}{1+t^2\mu(t)}}.$$

In order to prove that, every mapping $f$ defined by \eqref{convex} is a diffeomorphism we use Choquet-Kneser-Rado theorem. First of all $$\arg \partial_t F(z)= (\pi/2+t).$$ Therefore $$\partial_t \arg \partial_t F(z)=1>0$$ which means that $F(\mathbf{T})$ is a convex curve.

As $$\frac{\partial_t F(z)}{|\partial_t F(z)|}=iz,$$  if $z_1, z_2\in\mathbf{T}$ with $f(z_1)= f(z_2)$, then $$\frac{\partial_t F(z_1)}{|\partial_t F(z_1)|}=\frac{\partial_t F(z_2)}{|\partial_t F(z_2)|}$$ and so $z_1=z_2$. Thus by Choquet-Kneser-Rado theorem, $F$ is a diffeomorphism.

If $f$ is not $C^1$ up to the boundary, then we apply the approximating sequence. Let $\Omega$ be a fixed Jordan domain and assume that $\phi$ is a conformal mapping of the unit disk onto $\Omega$, with $\phi(0)=0$. For $r_n=\frac{n}{n+1}$, let $\Omega_n=\varphi(r_n\mathbf{U})$, and let $U_n=f^{-1}\Omega_n$. Let $\phi_n:\mathbf{U}\to U_n$ be a conformal mapping satisfying the condition $\phi_n(0)=0$. Then $f_n=f\circ \phi_n$ is a conformal mapping of the unit disk onto the Jordan domain $\Omega_n$. Further, by subharmonic property of $|\phi'(z)|$ we conclude that  $$R_n=|\partial \Omega_n|=\int_{\mathbf T}|\phi'(r_n z)|dz|<
\int_{\mathbf T}|\phi'( z)|dz|=|\partial\Omega|=R.$$

Then we have that
\begin{equation}\label{partibar21}|\partial f_n(z)|\le \frac{R_n}{1-|z|^2},\ \ \ z\in\mathbf{U}.\end{equation}

As $\phi_n$ converges in compacts to the identity mapping, and thus $\phi'_n$ converges in compacts to the constant $1$, we conclude that the inequality \eqref{partibar} is true for non-smooth domains.

It remains to consider the equality statement in this case. But we know that $\partial \Omega$ is rectifiable if and only if $\partial_t f\in h_1(\mathbf{U})$. (See e.g. \cite[Theorem 2.7]{kmm}). Here $h_1$ stands for the Hardy class of harmonic mappings. Now the proof is just repetition of the previous approach, and we omit the details.

\end{proof}
\begin{example}
If $\mu(z)=z^n$, then $F$ defined in \eqref{convex}, maps the unit disk to $n+2-$regular polygon of perimeter $2\pi R$ and centered at 0.
Namely we have that
$$\partial_z F(z) = \frac{R}{1+z^{n+2}},\ \ \ \partial_{\bar z} F(z)=\frac{Rz^{n}}{1+z^{n+2}}.$$ The rest follows from the similar statement obtained by Duren in \cite[p.~62]{duren}.
\end{example}

\begin{remark}
If $\mu$ is a holomorphic mapping of the unit disk onto itself and $F$ is defined by \eqref{convex}, then $F(0)=0$ and $$|DF|^2:=|F_z|^2+|F_{\bar z}|^2\ge \frac{R^2}{2} .$$
Indeed we have that $$|DF|^2=R^2\frac{1+|\mu|^2}{|1+z^2\mu|^2}\ge \frac{R^2}{2}=\frac{L^2}{8\pi^2}\ge \frac{\rho^2}{2}.$$ Here $\rho=\mathrm{dist}(0,\partial\Omega)$.
Thus we have the sharp inequality
 \begin{equation}\label{df2}|DF|^2\ge \frac{\rho^2}{2}.\end{equation}
In \cite{Dk} it is proved that we have the general  inequality \begin{equation}\label{df8}|Df|^2\ge \frac{\rho^2}{16},\end{equation} for every harmonic diffeomorphism of the unit disk onto a convex domain $\Omega$ with $f(0)=0$. Some examples suggest that the best inequality in this context is 
 \begin{equation}\label{df4}|Df|^2\ge \frac{\rho^2}{8},\end{equation}
 The last conjectured inequality is not proved. The gap between $\frac{\rho^2}{2}$ and $\frac{\rho^2}{8}$ in \eqref{df2} and \eqref{df4} appears as the mappings $F$ are special extremal mappings which for the case of $\Omega$ being the unit disk are just rotations.

\end{remark}

\end{document}